\documentclass[12pt]{article}

\usepackage{babel}
\usepackage[T1]{fontenc}
\usepackage{epsfig,amsmath,amssymb,amsthm,amsfonts}
\usepackage{lineno}
\usepackage[utf8]{inputenc}
\usepackage{natbib}
\usepackage{caption}
\usepackage{subcaption}
\usepackage{makecell}
\usepackage{multirow}
\DeclareMathOperator*{\argmax}{arg\,max}

\newtheorem{corollary}{Corollary}
\newtheorem{prop}{Proposition}

\newtheorem{definition}{Definition}
\newtheorem{rem}{Remark}
\newcommand{\piflat}{\pi^{(\text{flat})}}
\newcommand{\piprop}{\pi^{(\text{prop})}}
\newcommand{\pireal}{\pi^{(\text{real})}}
\newcommand{\wbird}{w^{(\text{bird})}}
\newcommand{\wspec}{w^{(\text{spec})}}
\newcommand{\wrare}{w^{(\text{rare})}}

\begin{document}
\newcommand{\ncm}{\newcommand}
\ncm{\bfm}[1]{\mbox{\boldmath $#1$}}
\ncm{\sbfm}[1]{\mbox{\scriptsize\boldmath $#1$}}
\ncm{\scr}[1]{\mbox{\scriptsize #1}}
\ncm{\scrmath}[1]{\mbox{\scriptsize $#1$}}
\ncm{\bfmscr}[1]{\mbox{\scriptsize{\boldmath $#1$}}}

\ncm{\R}{{\mathbb{R}}}
\ncm{\Z}{{\mathbb{Z}}}
\ncm{\T}{{\mathbb{T}}}
\ncm{\Smath}{{\mathbb{S}}}
\ncm{\N}{{\mathbb{N}}}
\ncm{\C}{{\mathbb{C}}}
\ncm{\A}{{\mathbb{A}}}
\ncm{\amath}{\bfm{a}}
\ncm{\V}{{\mathbb{V}}}
\ncm{\Hap}{{\mathbb{H}}}
\ncm{\MMD}{{\mathbb{MD}}}
\ncm{\E}{\mathbb{E}}
\ncm{\Prob}{\mathbb{P}}

\ncm{\sfS}{\mathsf{S}}

\ncm{\cA}{{\cal A}}
\ncm{\cB}{{\cal B}}
\ncm{\cC}{{\cal C}}
\ncm{\calF}{{\cal F}}
\ncm{\cD}{{\cal D}}
\ncm{\cG}{{\cal G}}
\ncm{\cL}{{\cal L}}
\ncm{\cN}{{\cal N}}
\ncm{\cI}{{\cal I}}
\ncm{\cJ}{{\cal J}}
\ncm{\cH}{{\cal H}}
\ncm{\cV}{{\cal V}}
\ncm{\cW}{{\cal W}}
\ncm{\cT}{{\cal T}}
\ncm{\cX}{{\cal X}}
\ncm{\cQ}{{\cal Q}}
\ncm{\cR}{{\cal R}}
\ncm{\cS}{{\cal S}}
\ncm{\cM}{{\cal M}}
\ncm{\cU}{{\cal U}}
\ncm{\cP}{{\cal P}}
\ncm{\cZ}{{\cal Z}}
\ncm{\cO}{{\cal O}}
\ncm{\cPzer}{{\cal P}_0}
\ncm{\cPone}{{\cal P}_1}
\ncm{\cPk}{{\cP_{\mbox{\scr{known}}}}}
\ncm{\cF}{{\cal F}}
\ncm{\cE}{{\cal E}}
\ncm{\cMD}{{\cal MD}}
\ncm{\tcV}{\tilde{\cal V}}
\ncm{\cCobs}{{\cal C}_{\scr{obs}}}

\ncm{\Om}{\Omega}
\ncm{\om}{\omega}
\ncm{\va}{\varepsilon}
\ncm{\vam}{\varepsilon_{\scr{max}}}
\ncm{\de}{\delta}
\ncm{\De}{\Delta}
\ncm{\ga}{\gamma}
\ncm{\Ga}{\Gamma}
\ncm{\la}{\lambda}
\ncm{\ka}{\kappa}
\ncm{\si}{\sigma}
\ncm{\Si}{\Sigma}
\ncm{\La}{\Lambda}
\ncm{\eps}{\epsilon}

\ncm{\bY}{\bfm{Y}}
\ncm{\bA}{\bfm{A}}
\ncm{\bB}{\bfm{B}}
\ncm{\bC}{\bfm{C}}
\ncm{\bD}{\bfm{D}}
\ncm{\bF}{\bfm{F}}
\ncm{\bI}{\bfm{I}}
\ncm{\bZ}{\bfm{Z}}
\ncm{\bG}{\bfm{G}}
\ncm{\bH}{\bfm{H}}
\ncm{\bL}{\bfm{L}}
\ncm{\bP}{\bfm{P}}
\ncm{\bQ}{\bfm{Q}}
\ncm{\bS}{\bfm{S}}
\ncm{\bT}{\bfm{T}}
\ncm{\bU}{\bfm{U}}
\ncm{\bM}{\bfm{M}}
\ncm{\bN}{\bfm{N}}
\ncm{\bR}{\bfm{R}}
\ncm{\bW}{\bfm{W}}
\ncm{\bX}{\bfm{X}}
\ncm{\bu}{\bfm{u}}
\ncm{\bv}{\bfm{v}}
\ncm{\bw}{\bfm{w}}
\ncm{\bwpr}{\bfm{w}^\prime}
\ncm{\bhp}{\bfm{h}^\prime}
\ncm{\bc}{\bfm{c}}
\ncm{\bd}{\bfm{d}}
\ncm{\bh}{\bfm{h}}
\ncm{\bm}{\bfm{m}}
\ncm{\bn}{\bfm{n}}
\ncm{\bb}{\bfm{b}}
\ncm{\bg}{\bfm{g}}
\ncm{\be}{\bfm{e}}
\ncm{\bl}{\bfm{l}}
\ncm{\bp}{\bfm{p}}
\ncm{\bq}{\bfm{q}}
\ncm{\br}{\bfm{r}}
\ncm{\bs}{\bfm{s}}
\ncm{\bx}{\bfm{x}}
\ncm{\by}{\bfm{y}}
\ncm{\bz}{\bfm{z}}
\ncm{\balp}{\bfm{\alpha}}
\ncm{\bbe}{\bfm{\beta}}
\ncm{\bxi}{\bfm{\xi}}
\ncm{\bth}{\bfm{\theta}}
\ncm{\bom}{\bfm{\om}}
\ncm{\bmu}{\bfm{\mu}}
\ncm{\bde}{\bfm{\de}}
\ncm{\bva}{\bfm{\va}}
\ncm{\beps}{\bfm{\eps}}
\ncm{\bga}{\bfm{\ga}}
\ncm{\bka}{\bfm{\ka}}
\ncm{\bla}{\bfm{\la}}
\ncm{\bpi}{\bfm{\pi}}
\ncm{\brho}{\bfm{\rho}}
\ncm{\boldeta}{\bfm{\eta}}
\ncm{\bphi}{\bfm{\phi}}
\ncm{\bLa}{\bfm{\Lambda}}
\ncm{\bPi}{\bfm{\Pi}}
\ncm{\bSi}{\bfm{\Si}}
\ncm{\bone}{\bfm{1}}

\ncm{\sbb}{\sbfm{b}}
\ncm{\sbc}{\sbfm{c}}
\ncm{\sbd}{\sbfm{d}}
\ncm{\sbm}{\sbfm{m}}
\ncm{\sbn}{\sbfm{n}}
\ncm{\sbC}{\sbfm{C}}
\ncm{\sbM}{\sbfm{M}}
\ncm{\sbX}{\sbfm{X}}
\ncm{\sbw}{\sbfm{w}}
\ncm{\sbx}{\sbfm{x}}
\ncm{\sby}{\sbfm{y}}
\ncm{\subu}{\scrmath{\ubu}}
\ncm{\subv}{\scrmath{\ubv}}
\ncm{\subw}{\scrmath{\ubw}}
\ncm{\subx}{\scrmath{\ubx}}
\ncm{\suby}{\scrmath{\uby}}
\ncm{\subX}{\scrmath{\ubX}}

\ncm{\hbe}{\hat{\beta}}
\ncm{\heta}{\hat{\eta}}
\ncm{\hth}{\hat{\theta}}
\ncm{\hbth}{\hat{\bth}}
\ncm{\hs}{\hat{s}}
\ncm{\hN}{\hat{N}}
\ncm{\hF}{\hat{F}}
\ncm{\hI}{\hat{I}}
\ncm{\hP}{\hat{P}}
\ncm{\hQ}{\hat{Q}}
\ncm{\htau}{\hat{\tau}}
\ncm{\hla}{\hat{\lambda}}
\ncm{\hmu}{\hat{\mu}}
\ncm{\hpi}{\hat{\pi}}

\ncm{\mast}{m^\ast}
\ncm{\cast}{c^\ast}
\ncm{\fast}{f^\ast}
\ncm{\siast}{\si^\ast}
\ncm{\psiast}{\psi^\ast}
\ncm{\tsiast}{\tilde{\si}^\ast}
\ncm{\alfast}{\alpha^\ast}
\ncm{\tkaast}{\tilde{\kappa}^\ast}
\ncm{\Xast}{X^\ast}
\ncm{\Yast}{Y^\ast}

\ncm{\ap}{a^\prime}
\ncm{\hp}{h^\prime}
\ncm{\ip}{i^\prime}
\ncm{\jp}{j^\prime}
\ncm{\kp}{k^\prime}
\ncm{\lp}{l^\prime}
\ncm{\np}{n^\prime}
\ncm{\npr}{n^\prime}
\ncm{\qp}{q^\prime}
\ncm{\rp}{r^\prime}
\ncm{\spr}{s^\prime}
\ncm{\up}{u^\prime}
\ncm{\vp}{v^\prime}
\ncm{\wpr}{w^\prime}
\ncm{\xp}{x^\prime}
\ncm{\yp}{y^\prime}
\ncm{\zp}{z^\prime}
\ncm{\Cp}{C^\prime}
\ncm{\Gp}{G^\prime}
\ncm{\Ip}{I^\prime}
\ncm{\Mp}{M^\prime}
\ncm{\Np}{N^\prime}
\ncm{\Npr}{N^\prime}
\ncm{\Tp}{T^\prime}
\ncm{\gap}{\ga^\prime}
\ncm{\phpr}{\phi^\prime}

\ncm{\wbis}{w^{\prime\prime}}

\ncm{\tih}{\tilde{h}}
\ncm{\tZ}{\tilde{Z}}
\ncm{\tA}{\tilde{A}}
\ncm{\tC}{\tilde{C}}
\ncm{\tD}{\tilde{D}}
\ncm{\tF}{\tilde{F}}
\ncm{\tI}{\tilde{I}}
\ncm{\tN}{\tilde{N}}
\ncm{\tY}{\tilde{Y}}
\ncm{\tmu}{\tilde{\mu}}
\ncm{\tOm}{\tilde{\Omega}}
\ncm{\tnu}{\tilde{\nu}}
\ncm{\tsi}{\tilde{\sigma}}
\ncm{\tal}{\tilde{\alpha}}
\ncm{\tbeta}{\tilde{\beta}}
\ncm{\tde}{\tilde{\delta}}
\ncm{\tka}{\tilde{\ka}}
\ncm{\txi}{\tilde{\xi}}
\ncm{\tmathV}{\tilde{\V}}
\ncm{\tV}{\tilde{V}}
\ncm{\tr}{\tilde{r}}
\ncm{\tu}{\tilde{u}}
\ncm{\tw}{\tilde{w}}
\ncm{\twpr}{\tilde{w}^\prime}
\ncm{\tb}{\tilde{b}}
\ncm{\td}{\tilde{d}}
\ncm{\tp}{\tilde{p}}
\ncm{\tf}{\tilde{f}}
\ncm{\tn}{\tilde{n}}
\ncm{\tR}{\tilde{R}}
\ncm{\tS}{\tilde{S}}
\ncm{\tL}{\tilde{L}}
\ncm{\tl}{\tilde{l}}
\ncm{\tP}{\tilde{P}}
\ncm{\tSmath}{\tilde{\mathbb{S}}}
\ncm{\tT}{\tilde{T}}
\ncm{\tK}{\tilde{K}}
\ncm{\tva}{\tilde{\va}}
\ncm{\tla}{\tilde{\la}}
\ncm{\tpi}{\tilde{\pi}}
\ncm{\trho}{\tilde{\rho}}
\ncm{\tbom}{\tilde{\bfm{\om}}}
\ncm{\tbxi}{\tilde{\bfm{\xi}}}
\ncm{\tbrho}{\tilde{\bfm{\rho}}}
\ncm{\tbg}{\tilde{\bg}}
\ncm{\tbb}{\tilde{\bb}}
\ncm{\tbr}{\tilde{\br}}
\ncm{\tbf}{\tilde{\bfm{f}}}
\ncm{\tbD}{\tilde{\bfm{D}}}
\ncm{\tbH}{\tilde{\bfm{H}}}
\ncm{\tbone}{\tilde{\bfm{1}}}
\ncm{\tbe}{\tilde{\bfm{e}}}

\ncm{\baW}{\bar{W}}
\ncm{\bacW}{\bar{\cW}}
\ncm{\bacV}{\bar{\cV}}
\ncm{\baf}{\bar{f}}
\ncm{\bah}{\bar{h}}
\ncm{\ban}{\bar{n}}
\ncm{\bap}{\bar{p}}
\ncm{\bav}{\bar{v}}
\ncm{\baw}{\bar{w}}
\ncm{\baZ}{\bar{Z}}
\ncm{\baY}{\bar{Y}}
\ncm{\baS}{\bar{S}}
\ncm{\baH}{\bar{H}}
\ncm{\baA}{\bar{A}}
\ncm{\baD}{\bar{D}}
\ncm{\baC}{\bar{C}}
\ncm{\baN}{\bar{N}}
\ncm{\baQ}{\bar{Q}}
\ncm{\bal}{\bar{l}}
\ncm{\bam}{\bar{m}}
\ncm{\bae}{\bar{e}}
\ncm{\bacR}{\bar{{\cal R}}}
\ncm{\bacP}{\bar{{\cal P}}}
\ncm{\babe}{\bar{\beta}}
\ncm{\baka}{\bar{\kappa}}
\ncm{\bamu}{\bar{\mu}}
\ncm{\banu}{\bar{\nu}}
\ncm{\bade}{\bar{\de}}
\ncm{\bala}{\bar{\la}}
\ncm{\baga}{\bar{\ga}}
\ncm{\barho}{\bar{\rho}}
\ncm{\babf}{\bar{\bfm{f}}}
\ncm{\babD}{\bar{\bfm{D}}}
\ncm{\babA}{\bar{\bfm{A}}}
\ncm{\babQ}{\bar{\bfm{Q}}}
\ncm{\babW}{\bar{\bfm{W}}}
\ncm{\babh}{\bar{\bfm{h}}}
\ncm{\babr}{\bar{\bfm{r}}}
\ncm{\babde}{\bar{\bfm{\de}}}
\ncm{\babrho}{\bar{\bfm{\rho}}}
\ncm{\babone}{\bar{\bfm{1}}}

\ncm{\chnu}{\check{\nu}}

\ncm{\uC}{\underline{C}}
\ncm{\ucH}{\underline{\cH}}
\ncm{\ucX}{\underline{\cX}}
\ncm{\ubx}{\underline{\bx}}
\ncm{\ubu}{\underline{\bu}}
\ncm{\ubv}{\underline{\bv}}
\ncm{\ubw}{\underline{\bw}}
\ncm{\ubX}{\underline{\bX}}
\ncm{\uby}{\underline{\by}}
\ncm{\ubY}{\underline{\bY}}

\ncm{\ocH}{\overline{\cH}}

\ncm{\sca}{\scr{a}}
\ncm{\scn}{\scr{n}}

\ncm{\Lin}{\, \stackrel{\cal L} \in}
\ncm{\Leq}{\, \stackrel{\cal L} =}
\ncm{\Lto}{\, \stackrel{\cal L} \longrightarrow}
\ncm{\pto}{\, \stackrel{p} \longrightarrow}
\ncm{\asto}{\, \stackrel{\rm a.s.} \longrightarrow}
\ncm{\Cov}{\mbox{Cov}}
\ncm{\Var}{\mbox{Var}}
\ncm{\sameord}{\stackrel{\cup}{{\scriptstyle \cap}}}

\ncm{\ith}{i^{\scr{th}}}
\ncm{\jth}{j^{\scr{th}}}
\ncm{\kth}{k^{\scr{th}}}
\ncm{\lth}{l^{\scr{th}}}
\ncm{\Bin}{\mbox{Bin}}
\ncm{\CV}{\mbox{CV}}
\ncm{\Exp}{\mbox{Exp}}
\ncm{\Hyp}{\mbox{Hyp}}
\ncm{\Po}{\mbox{Po}}
\ncm{\mm}{\mbox{mm}}
\ncm{\PD}{\mbox{PD}}
\ncm{\Ctot}{\bar{C}}
\ncm{\Ctottiny}{C_{\mbox{\tiny tot}}}
\ncm{\bzero}{\bfm{0}}
\ncm{\fappr}{\hat{f}}
\ncm{\bappr}{\hat{b}}
\ncm{\laappr}{\hat{\la}}
\ncm{\muappr}{\hat{\mu}}
\ncm{\pappr}{\hat{p}}
\ncm{\piappr}{\hat{\pi}}
\ncm{\kaappr}{\hat{\ka}}
\ncm{\Siappr}{\hat{\Si}}
\ncm{\bSiappr}{\hat{\bSi}}
\ncm{\demax}{\de_{\scr{max}}}
\ncm{\kamax}{\ka_{\scr{max}}}
\ncm{\mumin}{\mu_{\scr{min}}}
\ncm{\hmumin}{\hmu_{\scr{min}}}
\ncm{\Ias}{I_{\scr{as}}}
\ncm{\Ibas}{I_B}
\ncm{\cBall}{\cB_{\scr{all}}}
\ncm{\Inonas}{I_{\scr{nas}}}
\ncm{\Ilong}{I_{\scr{long}}}
\ncm{\Ishort}{I_{\scr{short}}}
\ncm{\cCcoarse}{\cC_{\scr{red}}}
\ncm{\cCmis}{\cC_{\scr{mism}}}
\ncm{\cCmishit}{\cC_{\scr{mismhit}}}
\ncm{\Cmax}{C_{\scr{max}}}
\ncm{\fmax}{f_{\scr{max}}}
\ncm{\Imin}{I_{\scr{min}}}
\ncm{\pmin}{p_{\scr{min}}}
\ncm{\Pmin}{P_{\scr{min}}}
\ncm{\thmin}{\theta_{\scr{min}}}

\ncm{\beq}{\begin{equation}}
\ncm{\eeq}{\end{equation}}
\ncm{\beqr}{\begin{eqnarray}}
\ncm{\eeqr}{\end{eqnarray}}
\ncm{\beqrn}{\begin{eqnarray*}}
\ncm{\eeqrn}{\end{eqnarray*}}
\ncm\rthm[1]{\ref{#1}}
\ncm\lb[1]{\label{#1}}
\ncm\re[1]{(\ref{#1})}
\ncm{\slut}{
  {\unskip\nobreak\hfill\penalty100\hskip1em\vadjust{}\nobreak
  \hfill\mbox{$\Box$}\parfillskip=0pt\finalhyphendemerits=0}}

\parindent=0mm
\newcommand*\samethanks[1][\value{footnote}]{\footnotemark[#1]}

\begin{titlepage}

\title{Classification Under Partial Reject Options}

\author{M{\aa}ns Karlsson\footnote{Department of Mathematics, Stockholm University, 106 91 Stockholm, Sweden. Email: mansk@math.su.se.} \and Ola H\"{o}ssjer\footnote{Department of Mathematics, Stockholm University, 106 91 Stockholm, Sweden. Email: ola@math.su.se.}}

\maketitle

\begin{abstract}
\parindent=0pt
We study set-valued classification for a Bayesian model where data originates from one of a finite number $N$ of possible hypotheses. Thus we consider the scenario where the size of the class\-ified set of categories ranges from 0 to $N$. Empty sets corresponds to an outlier, size 1 represents a firm decision that singles out one hypotheses, size $N$ corresponds to a rejection to classify, where\-as sizes $2\ldots,N-1$ repre\-sent a partial rejection, where some hypotheses are excluded from further analysis. We introduce a general framework of reward functions with a set-valued argument and derive the correspon\-ding optimal Bayes classifiers, for a homogeneous block of hypotheses and for when hypo\-theses are partitioned into blocks, where ambiguity within and between blocks are of different severity. We illustrate classification using an ornitho\-logical dataset, with taxa partitioned into blocks and parame\-ters estimated using MCMC. The associa\-ted reward function's tuning para\-me\-ters are chosen through cross-validation.


\par\bigskip
{{\bf Keywords:} }  Blockwise cross-validation, Bayesian classification, con\-formal prediction, classes of hypotheses, indifference zones, Markov Chain Monte Carlo, reward functions with set-valued inputs, set-val\-ued classi\-fiers.
\end{abstract}

\end{titlepage}

\section{Introduction}

Classification of observations among a finite number $N$ of hypotheses or categories is a well studied problem in statistics, and also a central concept of modern machine learning \citep{bishop2006pattern, hastie2009elements}. With a Bayesian approach the maximum aposteriori classifier maximizes the proba\-bili\-ty of a correct classification \citep{berger2013statistical, bishop2006pattern}. For some problems, when there is much ambi\-gui\-ty about the category that fits data the best, it is possible to add a reject option, where no classification is made \citep{chow1970optimum, ripley2007pattern, freund2004generalization, herbei2006classification}, and if such a reject option is followed by additional data collection, this leads to Bayesian sequential analysis \citep{arrow1949bayes}. These types of classifiers were generalized in \citet{karlsson2021identification}, and formulated in the context of reward functions with a set-valued argument. Each such reward function leads to a set-valued classifier, where the size of the classified set of categories is 0 if the classifier rejects all hypotheses, 1 for a firm decision that singles out one category, between 1 and $N$ for a partial reject option, where ambiguity remains between some but not all categories, and $N$ for a rejection to classify, i.e when none of the categories are singled out.
\par\medskip
In this paper we analyze reward functions with a set-valued argument in more detail, and in particular we derive the associated optimal Bayes classifiers that maximize expected rewards. We start by considering reward functions for a homogeneous collection of categories of the same type, and find explicit formulas for the Bayes classifier in terms of the ordered aposteriori probabili\-ties of all hypotheses and a penalty term for the size of the classified set. It is shown that conformal prediction \citep{vovk2005algorithmic, shafer2008tutorial} 
with posterior probabilities as nonconformity measure is a special case of our approach. 
Then we consider scenarios where the categories are divided into a number of blocks, such that categories (typically) are more similar within than between blocks. The associated Bayes classifiers involve the ordered aposteriori pro\-babilities within each block and penalty terms $a$ and $b$ for adding classified categories from the correct and wrong blocks respectively. In particular, we demon\-strate that classifi\-cation with indifference zones \citep{bechhofer1954single, goldsman1986tutorial} can be represen\-ted as an instance of set-valued classification with two blocks of categories. 
\par\medskip
Our framework of set-valued classification, with the set of categories parti\-tioned into blocks, is applied to an ornithological data set. Each observation in data consists of three observed traits for a bird and which taxon the bird belongs to. Based on these data we want to train a classifier of future birds to taxon. The taxa are partitioned into blocks with regard to cross-taxon similarities. The parameters of the underlying statistical model are estimated through a Markov Chain Monte Carlo procedure, and the tuning parameters of the set-valued reward function are estimated through cross validation.
\par\medskip
The article is organized as follows: In Section \ref{Sec:ClassSetRew} we introduce the statistical model with $N$ hypotheses, and define the optimal (Bayes) set-valued class-ifier. Then we introduce a large class of reward functions for models with one block of categories (Section \ref{Sec:One}) and several blocks of categories (Section \ref{Sec:Several}) respectively, and give explicit expressions for the corresponding Bayes classi\-fiers. The ornithological dataset is analyzed in Section \ref{Sec:Data} and a discussion in Section \ref{Sec:Disc} concludes the paper.

\section{Statistical model and optimal classifiers}\lb{Sec:ClassSetRew}

Consider a random variable $Z\in\cZ$, whose distribution follows one of $N$ possible hypotheses (or categories)
\beq
H_i: Z\sim f_i, \quad i=1,\ldots,N,
\lb{Hi}
\eeq
where $f_i$ is the density or probability function of $Z$ under $H_i$. We will assume a Bayesian framework, and thus the true but unknown hypothesis $I\in\cN=\{1,\ldots,N\}$ is a random variable. It is given a categorical prior distribution with parameter vector $\pi_i=P(I=i)$, for $i=1,\ldots,N$ and the corresponding parameter vector of the categorical posterior distribution of an observed value $z$ of $Z$ is
\beq
p_i = p_i(z) = \Prob(I=i|Z=z) = \frac{\pi_i f_i(z)}{\sum_{j=1}^N \pi_j f_j(z)}.
\lb{pi}
\eeq
Our objective is to classify $z$. To this end, a classifier $\hI=\hI(z)\subset \cN$ with partial reject options is defined as a subset of categories. When $\hI=\{i\}$, a firm decision of category $i$ is made, whereas $\hI=\cN$ corresponds to a reject option, where no decision is made about which categories that conform with $z$ the most. Note that it is the action ``to classify'' that is rejected, when $\hI = \cN$. The intermediate case $2\le |\hI| =m\le N-1$ corresponds to a partial reject option, where the classifier excludes $N-m$ hypotheses but rejects discrimination among the remaining $m=m(z)$ categories. Another possibility, $\hI=\emptyset$, corresponds to a scenario where none of the $N$ categories fit the observed data $z$ well, i.e. we exclude all hypotheses. This can be regarded as a safeguard against outliers or otherwise faulty data \citep{ripley2007pattern, karlsson2021identification}.
\par\medskip
Let 
\beq
p_{(1)}\le \ldots \le p_{(N)}
\lb{pOrdered}
\eeq
refer to the ordered posterior category probabilities. The well-known Max\-imum Aposteriori (MAP) classifier \citep{berger2013statistical}
\beq
\hI = \hI(z) = \{(N)\}
\lb{hIMAP}
\eeq
always makes a firm decision, so that 
\beq
\Prob(|\hI|=1)=1.
\lb{hISizeOne}
\eeq
It also maximizes the probability $\Prob(\hI=\{i\})$ of a correct classification, but has a higher probability of misclassification than correct classification, when $p_{(N)} < 1/2$. The MAP classifier can be formulated as $\hI$ being the classifier that maximizes the expected value 
\beq
 \bar{V} =\E[R(\hI,I)]
\lb{V}
\eeq
of the reward function
\beq
R(\cI,i) = \left\{\begin{array}{ll}
1, & \cI = \{i\},\\
0, & \cI \ne \{i\}
\end{array}\right.     
\lb{RMAP}
\eeq
assigned to a classified subset $\cI\subset\cN$ of categories when the true category is $i$. Since $\hI=\hI(Z)$, the expectation in \re{V} is with respect to $Z$ and $I$. Let the value function 
$$
V(z;\cI)=\E\left[R(\cI,I)\mid Z=z\right]
$$ 
refer to the conditional expected reward given $Z=z$. It is clear the optimal Bayes classifier, that maximizes the expected reward, or equivalently max\-imizes the expected value function $\bar{V}  = \E[V(Z;\hat{I}(Z))]$ is obtained as
\begin{align}
 \hI(z) &= \argmax_{\cI \subset \cN} V(z;\cI) \nonumber \\
 &= \argmax_{\cI \subset \cN} \E\left[R(\cI,I)\mid Z=z\right] \nonumber \\
 &= \argmax_{\cI \subset \cN} \sum_{i=1}^N R(\cI,i) p_i(z)
 \lb{Vz}
\end{align}
for each $z\in\cZ$. Since the reward function \eqref{RMAP} only takes non-zero values for singelton choices of $\cI\subset\cN$, we only need to consider $\cI = \{i\}$ for $i=1,\ldots,N$, in order to find $\hI$. For each choice $\cI = \{i\}$, $V(z;\{i\}) = p_i$. Thus we choose $i = (N)$, since $p_{(N)}$ for each fixed $z$ is the largest value $V(z,\cI)$ can return. This shows that the MAP classifier \re{hIMAP} is the optimal classifier under \eqref{RMAP}.
\par\medskip
In this paper we will consider optimal classifiers \eqref{Vz}
for reward functions other than \re{RMAP}; those that allow not only for single classified categories, but also for outliers and (partial) reject options.  

\section{Partial rejection, one block of categories}\lb{Sec:One}

When $\cN$ represents one homogeneous block of categories, the following class of reward functions is natural to use: 
\par\medskip
\begin{definition} \label{def1}
Given a set of categories $\cN=\{1,\ldots,N\}$, with $i \in \cN$ the true category of an observation and $\cI \subset \cN$ the classified subset of categories, a reward function
 \beq
 R(\cI,i) = R(\tau(\cI),\tau(i))
 \lb{RInv}
 \eeq
 invariant w.r.t. permutations $\tau:\cN\to\cN$ of the labels is called an \emph{invariant reward function}.
\end{definition}
It is easily seen that
\beq
R(\cI,i) = 1(i\in\cI) - g(|\cI|)
\lb{Rg}
\eeq
is an invariant reward function. The first term of \re{Rg} corresponds to a reward of 1 for a classified set $\cI$ that includes the true category $i$, whereas the second term $g(|\cI|)\ge 0$ penalizes based on the size $|\cI|$ of the classified set. Of particular interest are the following classifiers:
\begin{definition}
A classifier containing the $m \in \{0,\ldots,N\}$ categories with the largest post\-erior probabilities, where
\beq
\hI_{m} = \begin{cases}
                       \{(N+1-m),\ldots,(N)\}, &\quad m \ge 1, \\
                       \emptyset, &\quad m=0,
                      \end{cases}
\lb{hInz}
\eeq
is called an \emph{$m$ most probable classifier}, abbreviated as an $m$-MP classifier. 
\end{definition}
The MAP classifier is thus a special case of the more general class of $m$-MP classifiers, and obviously the 1-MP classifier is the MAP classifier, whereas the $0$-MP classifier corresponds to the empty set, since we then choose none of the categories. We will now link invariant reward functions of type \re{Rg} to $m$-MP classifiers: 
\par\medskip
\begin{prop}\lb{Prop:One} The optimal classifier for an invariant reward function of type \re{Rg}, is an $m$-MP classifier $\hI(z) = \hI_{m(z)}$ with 
\beq
m(z) = \argmax_{0\le m \le N} \left[v(m;z) - g(m)\right].
\lb{nz}
\eeq
where $v(0;z)=0$ and 
\beq
v(m;z) = \sum_{j=1}^m p_{(N+1-j)}, \quad m=1,\ldots,N.
\lb{vnz}
\eeq
\end{prop}
\begin{proof}
Recall that the optimal classifier $\hI(z)$ maximizes, for each $z\in\cZ$, the value function $V(z;\cI)$ among all nonempty subsets $\cI$ of $\cN$. For the reward function of equation \re{Rg} we have that 
\begin{align}
V(z;\cI) = \E\left[ R(\cI,I) \mid Z=z\right] &= \sum_{i=1}^N \left(1(i\in\cI) - g(|\cI|)\right)p_i(z) \nonumber\\ 
&= \sum_{i=1}^N 1(i\in\cI)p_i(z) - \sum_{i=1}^N g(|\cI|)p_i(z) \nonumber\\
&= \sum_{i\in\cI} p_i - g(|\cI|)\nonumber\\
&\le  \sum_{j=1}^{|\cI|} p_{(N+1-j)} - g(|\cI|)\nonumber\\
&= v(|\cI|;z) - g(|\cI|).
\lb{VzProof}
\end{align}
Note that the inequality occurs since we go from considering a subset $\cI \subseteq \cN$ of categories to considering another subset  with equally many but the most probable categories. Consequently, among all $\cI\subset\cN$ of size $|\cI|=m$, the value function $V(z;\cI)$ is maximized by $\hI_m$ in \re{hInz}, for some $m\in\{0,\ldots,N\}$. Among these subsets, the optimal classifier is $\hI=\hI_{m(z)}$, where $m(z)$ is the value of $|\cI|$ that maximizes the right hand side of \re{VzProof}. Since this value of $m(z)$ is identical to the one in \re{nz}, this finishes the proof.
\end{proof}
\par\medskip
{\bf Example 1} (Classification with reject options.) \citet{ripley2007pattern} introduced a reward function with a reject option. In our notation it corre\-sponds to 
\beq
R(\cI,i) = \left\{\begin{array}{ll}
1(i=j), & \cI=\{j\},\\
r, & \cI = \cN,\\
0, & |\cI|\in \{0,2,3,\ldots,N-1\},
\end{array}\right.
\lb{RRipley}
\eeq
with a reward of $1/N < r < 1$ assigned to the reject option. The special case of \re{RRipley} for $N=2$ categories was treated by \citet{chow1970optimum, herbei2006classification}. Note that \re{RRipley} is equivalent to a reward function \re{Rg} with penalty term
\beq
g(m) = \left\{\begin{array}{ll}
0, & m=1,\\
1, & m =2,\ldots,N-1,\\
1-r, & m=N. 
\end{array}\right. 
\lb{gRipley}
\eeq 
The corresponding classifier
\beq
\hI = \left\{
\begin{array}{ll}
(N), & p_{(N)} > r,\\
\cN, & p_{(N)} \le r
\end{array}\right.
\lb{hIRipley}
\eeq
is of interest for Bayesian sequential analysis \citep{arrow1949bayes, berger2013statistical} and sequential clinical trials \citep{carlin1998approaches}, where the reject option $\hI=\cN$ corresponds to delaying the decision and collecting more data before the seqential procedure is stopped and a specific category is chosen. 
\hfill\slut
\par\medskip
The next result is a corollary of Proposition \ref{Prop:One}, and it treats an important class of reward functions with a convex penalty function: 
\par\medskip
\begin{corollary}\lb{Cor:Conv}
Consider a classifier $\hI(z)$ based on a reward function \re{Rg}, for which the penalty term $g(m)$ is a convex function of $m$, with $g(0)=0$. Then \re{nz} simplifies to $\hI(z)=I_{m(z)}$, with 
\beq
m(z) = \max\{m^\prime(z),m^{\prime\prime}(z)\},
\lb{nzLin}
\eeq
where 
\beq
\begin{array}{rcl}
m^\prime(z) &=& 0,\\
m^{\prime\prime}(z) &=& \max\{1\le m \le N; \, p_{(N+1-m)} \ge g(m)-g(m-1)\},
\end{array}
\lb{nzConv}
\eeq
and $\max \emptyset = -\infty$ in the definition of $m^{\prime\prime}(z)$. 
\end{corollary}
\begin{rem}
 Note that in the second line of \eqref{nzConv}, the inequality constitutes an inclusion criterion that a category needs to fulfill to be included in the classifier. As can be seen from the left and right hand sides of the inequality, the category with the largest posterior probability not yet included, will be included if its posterior probability is larger than the added penalty $g(m)-g(m-1)$ associated with enlarging the size of the classifier from $m-1$ to $m$.
\end{rem}

\begin{proof}
 In order to prove \re{nzConv}, we first deduce from \re{pOrdered} that $v(m;z)$ is a concave function of $m$ (indeed, the differences $v(m+1;z) - v(m,z)$ are decreasing as $m$ increases and thus $v(m;z)$ is concave). Consequently, if $g(m)$ is a convex function of $m$ it follows that $v(m;z)-g(m)$ is a concave function of $m$, and it is therefore maximized by 
\beq   
m(z) = \max \left\{0, \max_{1\le m \le N}\left\{\, v(m;z)-g(m) \ge v(m-1;z)-g(m-1)\right\}\right\},
\lb{mzProof}
\eeq
where $\max \emptyset = -\infty$ in the inner maximization. By the definition of $v(m;z)$ in \re{vnz}, equation \re{mzProof} is equivalent to the expression for $m(z)$ given in equation \re{nzConv}. 
\end{proof}
\par\medskip
{\bf Example 2} (Proportion-based reward functions.) In order to illustrate Corollary \ref{Cor:Conv} we introduce the following class of reward functions with a penalty germ $g(m) = c\max(0,m-1)$ that only includes one cost parameter $c$:  
\begin{definition}
An invariant reward function of the form
 \beq
 R(\cI,i) = 1(i\in\cI) - c\max(0,|\cI|-1),
 \lb{Rc}
 \eeq
with $c\ge0$, is called a \emph{proportion-based reward function}.
\end{definition}
Note that a proportion-based reward function has a penalty term that is almost linear in $|\cI|$. It is only the maximum operator that prevents us from refering to \re{Rc} as a reward function with a linear penalty term (abbreviated as a linear reward function). We may interpret $c$ as a cost per {\it extra} classified category on top of the first one. For proportion-based reward functions, Corollary \ref{Cor:Conv} simplifies as follows:
\begin{corollary}
A proportion-based reward function \eqref{Rc} gives rise to an $m$-MP classifier \re{hInz} with $m(z)$ as in \re{nzLin} and 
\beq
\begin{array}{rcl}
m^\prime(z) &=& 1,\\
m^{\prime\prime}(z) &=& \max\{2\le m \le N; \, p_{(N+1-m)} \ge c\}.
\end{array}
\lb{nzLin2}
\eeq
\end{corollary}
\begin{proof}
Since a proportion-based reward function \re{Rc} has penalty term \\$g(m)=c\max(0,m-1)$, it follows that 
\beq
g(m)-g(m-1) = \left\{\begin{array}{ll}
0, & m=1, \\
c, & m=2,\ldots,N.
\end{array}\right.
\lb{gDiff}
\eeq
Inserting \re{gDiff} into \re{nzConv}, equation \re{nzLin2} follows. 
\end{proof}
\par\medskip
Note that the MAP classifier \re{hIMAP} is obtained for $c>p_{(N)}$. For this reason \citet{karlsson2021identification} restricted the cost term of \re{Rc} to a range $0\le c \le p_{(N)}$ and reparametrized it as 
\beq
c = \rho p_{(N)},
\lb{c}
\eeq
with $0\le \rho \le 1$. The $m$-MP classifier \re{hInz},
with $m=m(z)$ as in \re{nzLin} and \re{nzLin2}, then takes the form
\begin{equation}
\hI = \{i; \, p_i \ge \rho p_{(N)}\}.
\label{Irho}
\end{equation}
\hfill\slut
\par\medskip
{\bf Example 3} (Conformal prediction)
Conformal prediction \citep{vovk2005algorithmic, shafer2008tutorial} is a general method for creating a prediction region $\Gamma^\delta=\Gamma^\delta(z)$ for an observation $z$ with a confidence level $(1-\delta)$\%, where $\delta \in (0,1)$ is chosen freely, typically close to 0. For the case of categorical prediction, the conformal algorithm \citep[Section 4.3]{shafer2008tutorial} uses as input the new observation $z$ that we want to classify, a labeled training data set $B$, the number $\delta$, and a non\-conformity measure $A$. Then for each possible label $i\in\cN$ of $z$ a decision is made as to whether $i$ should be included in the prediction region $\Gamma^\delta$ or not.
\par\medskip
It turns out that conformal prediction is as a special case of the classification theory presented this section, when posterior probabilities are used as non\-conformity measure. We choose the penalty term $g(|\cI|) = c|\cI|$ in \eqref{Rg} and find that $\Gamma^\delta=\hI(z) = \hI_{m(z)}$, where
\begin{align}
m(z) &= \max \left\{ 0, \max \left\{ 1 \le m \le N; p_{(N+1-m)}\ge c\right\} \right\}\nonumber \\ &= \left\{i; 1\le i \le N, p_i \ge c\right\},
\label{conpred}
\end{align}
We can use theory from \citet{shafer2008tutorial} to specify how $c=c(\delta)$ in \eqref{conpred} relates to $\delta$. To this end, let $z$ be the observation we want to classify. For each possible label (or category) $i \in \cN$ of $z$, provisionally set $(z,i)$ as a future observation that is part of training data (although in practice $z$ is rather a future observation that we want to classify).  The batch of previous examples $B$ corresponds to a very large training data set (which is called $\cD=\cD_1\cup \ldots \cup \cD_N$ in Section \ref{Sec:Model}), with $n_i \to \infty$ observations from each category $i$, which in the limit makes it possible to know the distributions $f_1,\ldots,f_N$ of data under the various hypotheses. If $n_i/\sum_{j=1}^N n_j \to \pi_i$ when $n_i\to\infty$ for $i=1,\ldots,N$, the asymptotic batch corresponds to a mixture
\begin{equation}
f \sim \sum_{i=1}^N \pi_i f_i
\label{mixture}
\end{equation}
of the distributions of the $N$ categories. We then use the posterior probabi\-lities $\pi_i$ to define the nonconformity measure
\begin{equation}
A[B,(z,i)] = p_i(z).
\label{A}
\end{equation}
Let $F_i$ refer to the distribution function of $p_i(Z)$ when $Z \sim f_i$, and let
\begin{equation}
 F = \sum_{i=1}^N \pi_i F_i
\end{equation}
be the corresponding mixture distribution function. The conformal algo\-rithm then corresponds to choosing a classifier \eqref{conpred} with
\begin{equation}
 c = F^{-1}(\delta),
 \label{quant}
\end{equation}
where, as mentioned, $1-\delta$ is the confidence level of the prediction region.
\par\medskip
As a last remark, if the penalty parameter is changed from $c$ to $\rho=c/p_{(N)}$ (cf. Example 2), we notice that \eqref{conpred} is equivalent to \eqref{Irho} for \textit{any} value of $\rho \ge 0$.
\hfill\slut

\section{Several blocks of categories}\lb{Sec:Several}
This section will cover an extension of the classical classification problem, where observations belong to a category and categories belong to super\-categories, or blocks, as we will call them. Thus we partition the $N$ categories into $K$ blocks of sizes $N_1,\ldots,N_K$, with $\sum_{k=1}^K N_k = N$. Without loss of generality, the labels $i$ are defined so that each block
\beq
\cN_k = \left\{\sum_{l=1}^{k-1} N_l + 1 \le i \le \sum_{l=1}^k N_l \right\}
\lb{cNk}
\eeq
consists of adjacent categories. The different scenarios when it is worse to misclassify within a block than between blocks, and vice versa, will be a subject of study later on in the paper. In order to define classifiers that take this into account, we introduce a new type of reward functions.
\begin{definition} \label{def4}
 For a set of $N$ categories, partitioned into blocks $\cN_k$, $k=1,\ldots,K$ with $N_k$ categories in each block, \emph{block invariant reward functions} satisfy \eqref{RInv} only for block preserving permuations $\tau(\cN_k) = \cN_k$, $k=1,\ldots,K$.
\end{definition}
A class of block invariant reward functions is
\beq
R(\cI,i) = 1(i\in\cI) - g_{k(i)}(|\cI_{k(i)}|,|\cI|-|\cI_{k(i)}|),
\lb{Rgk}
\eeq
where 
\beq
\cI_k = \cI \cap \cN_k
\lb{cIk}
\eeq
contains the categories of the classified set $\cI$ that belong to block $k$, whereas $k(i)$ is the block to which $i$ belongs, i.e.\ $i\in \cN_{k(i)}$. Moreover, $g_k$ is a penalty term for misclassification, when the true category $i$ belongs to $\cN_k$. This term is a function of the number of categories $|\cI_{k(i)}|$ in the classified set $\cI$ that belong to the correct block as well as the number of classified categories $|\cI|-|\cI_{k(i)}|$ that belong to any of the wrong blocks. In order to define the classifiers of interest we order the posterior probabilities $p_i$, $i\in\cN_k$ within each block as
\begin{equation*}
 p_{(k1)} \le \ldots \le p_{(kN_k)}.
\end{equation*}
\begin{definition}
 For an integer vector
\begin{equation}
 \bm = (m_1,\ldots,m_K) \in \bigotimes_{k=1}^K \{0,\ldots,N_k\}
 \lb{bn}
\end{equation}
let 
\begin{equation}
 \hI_{\sbm} = \left\{(ki); \, 1\le k \le K, N_k+1-m_k\le i \le N_k\right\}
\lb{Ibnz}
\end{equation}
be a classifier that includes the $m_k$ categories with the largest posterior pro\-babilites from block $k$. We call this a \emph{composite classifier}.
\end{definition}
A composite classifier is a subset of $\cN$ of size 
\begin{equation}
\lb{nzSum}
m=\sum_{k=1}^K m_k.
\end{equation}
In particular, the two extreme scenarios with no categories classified or a rejection to classify, correspond to $\hI_{(0,\ldots,0)}=\emptyset$ and $\hI_{(N_1,\ldots,N_K)}=\cN$ respective\-ly. Moreover, \re{nzSum} reduces to an $m_k$-MP classifier \re{hInz} for a particu\-lar block $k$ if $\bm$ has only one non-zero element $m_k$ (cf. Example 4). The following result links composite classifiers to block invariant reward functions of type \re{Rgk}:
\par\medskip
\begin{prop}\lb{Prop:Several}
The optimal classifier, for a reward function \re{Rgk}, is a composite classifier $\hI(z)=\hI_{\sbfm{m}(z)}$ in \re{Ibnz}, with 
\beq
\bm(z) = \argmax_{\sbfm{m}=(m_1,\ldots,m_K)} \sum_{k=1}^K \left[v_k(m_k;z) - P_k g_k(m_k,m-m_k)\right],
\lb{nkz}
\eeq
where 
\begin{align}
P_k = P_k(z) &= \sum_{i\in\cN_k} p_i, \nonumber\\
m &= \sum_{k=1}^K m_k, \lb{vk} \\ 
v_k(m_k;z) &= \sum_{j=1}^{m_k} p_{(k,N_k+1-j)} \nonumber
\end{align}
and $v_k(0;z) = 0$ for any $k$.
\end{prop}
\begin{proof}
 The proof mimics that of Proposition \ref{Prop:One}. We start by finding the value function $V(z;\cI)$ in \re{Vz} for the block invariant reward function \re{Rgk}. It is given by  
\begin{align}
V(z;\cI) &=  \sum_{k=1}^K \sum_{i\in\cI_k} p_i - \sum_{k=1}^K P_k g_k(|\cI_k|,|\cI|-|\cI_k|) \nonumber\\
&\le  \sum_{k=1}^K \left[ \sum_{j=1}^{|\cI_k|} p_{(k,N_k+1-j)} - P_k  g_k(|\cI_k|,|\cI|-|\cI_k|)\right] \nonumber \\
&= \sum_{k=1}^K \left[ v_k(|\cI_k|;z) - P_k g_k(|\cI_k|,|\cI|-|\cI_k|)\right].
\lb{VzProof2}
\end{align}
From this it follows that $V(z;\cI)$ is maximized, among all $\cI$ with $|\cI_k|=m_k$ for $k=1,\ldots,K$, by $\hI_{(m_1,\ldots,m_K)}$ in \re{Ibnz}. The optimal classifier is therefore $\hI=\hI_{\sbm(z)}$, where $\bm(z)=(m_1(z),\ldots,m_K(z))$ is the value of $(|\cI_1|,\ldots,|\cI_K|)$ that maximizes the right hand side of \re{VzProof2}. Hence $\bm(z)$ is given by \re{nkz}.
\end{proof}
\par\medskip
{\bf Example 4} (Composite proportion-based reward functions.) We will con\-sider a class of reward functions that are special cases \re{Rgk}. These reward functions involve two cost parameters $a$ and $b$:
\par\medskip
\begin{definition}
 A block invariant reward function of the form
\begin{equation}
R(\cI,i) = 1(i\in\cI) - a\max\left(|\cI_{k(i)}|-1,0\right) - b\left(|\cI|-|\cI_{k(i)}|\right),
\lb{Rgkab}
\end{equation}
 where $0\le a \le b$ are fixed constants, is called a \emph{composite proportion-based reward function}.
\end{definition}
\par\medskip
It follows from Propositon \ref{Prop:Several} that composite proportion-based reward func\-tion have very explicit optimal classifiers:
\par\medskip
\begin{corollary} \label{cor:3}
A composite proportion-based reward function \eqref{Rgkab}, gives rise to an optimal classifier that is a composite classifier $\hI(z)=\hI_{\sbfm{m}(z)}$ in \re{Ibnz} with 
\begin{align}
m_k(z) &= \argmax_{0\le m_k\le N_k} \left[v_k(m_k,z) - a P_k\max(m_k-1,0) - b (1-P_k)m_k \right] \nonumber\\
&=  \max\{m_k^\prime(z),m_k^{\prime\prime}(z)\},
\lb{nkzab}
\end{align}
for $k=1,\ldots,K$, where
\beq
\begin{array}{rcl}
m_k^\prime(z) &=& 1\left(p_{(kN_k)}\ge (1-P_k)b\right),\\ 
m_k^{\prime\prime}(z) &=& \max\{2\le m_k \le N_k; \, p_{(k,N_k+1-m_k)} \ge P_k a + (1-P_k)b\},
\end{array}
\lb{mkprimbis}
\eeq
and with $\max \emptyset = -\infty$ used in the definition of $m_k^{\prime\prime}(z)$. 
\end{corollary}
\begin{proof}
 In view of Proposition \ref{Prop:Several}, it suffices to prove that for a composite proportion-based reward function \re{Rgkab}, the optimal classifier is of type $\hI(z)=\hI_{\sbfm{m}(z)}$, where $\bm(z)=(m_1(z),\ldots,n_K(z))$ is given by \re{nkzab}. To this end, we first notice, from the right hand side of \re{VzProof2}, that the value function of a classified set $\hI_{\sbm}$ is
\begin{align}
V(z;\hI_{\sbm}) &= \sum_{k=1}^K \left[ v_k(m_k;z) - P_k \left(a\max\{m_k-1,0\} + b\sum_{l;l\ne k} m_l\right)\right] \nonumber\\
&= \sum_{k=1}^K \left[ v_k(m_k;z) - \left(P_k a\max\{m_k-1,0\} + (1-P_k) b m_k\right) \right].
\lb{VzSeveral}
\end{align}
Since $V(z;\hI_{\sbm})$ splits into a sum of $K$ terms that are functions of $m_1,\ldots,m_K$ respectively, it follows that $V(z;\hI_{\sbm})$ is maximized, as a function of $\bm$, by maximizing each term separately with respect to $m_k$. The maximum for term $k$, on the right hand side of \re{VzSeveral}, is attained for
\begin{equation}
 m_k(z) = \argmax_{0\le m_k \le N_k} \left[ v_k(m_k;z) - \left(P_k a\max\{m_k-1,0\} + (1-P_k) b m_k\right) \right],
\lb{nkzProof}
\end{equation}
in accordance with the first identity of \re{nkzab}. The second identity of \re{nkzab} follows from the fact that the function being maximized in \re{nkzProof} is concave in $m_k$.
\end{proof}

To give some more intuition to the choice of $a$ and $b$ for the reward function \re{Rgkab}, we will look at the penalty term 
\begin{equation*}
 g_{k(i)}(\cI) = a\max(|\cI_{k(i)}|-1,0) + b(|\cI|-|\cI_{k(i)}|)
\end{equation*}
and the optimal classifier $\hI(z)=\hI_{\sbfm{m}(z)}$ defined in equations \re{nkzab}-\re{mkprimbis} of Corollary \ref{cor:3}. Note that a cost of $a$ is incurred per {\it extra} category from the correct block in the classified set $\hI$, whereas a cost of $b$ is added for each category in the classified set originating from the wrong block. These costs are chosen, and can be interpreted as threshold values for including categories in the classifier, especially when looking at \eqref{mkprimbis}. From the first row of this equation we notice that a low value on $b$ means that we are more prone to include the most probable category from each block, whereas the second row implies that with a low value of $a$ we are more prone to include several categories from the same block. However, since $b$ occurs in the second row of \eqref{mkprimbis} as well, a small $a$ might not have any effect if $b$ is large. On the other hand, combining a small $b$ with a large $a$, we get a classifier that is composed of categories from many blocks, but few categories from each block. Such a classifier might not include the correct category, but will to a large extent include some category from the correct block, and might be suitable when we want to safeguard in particular against erraneous superclassification. Finally, if we want to ensure that $\hI \ne \emptyset$ in composite classification, we have to choose 
\begin{equation}
 b \le \max\left(\frac{p_{(1N_1)}}{1-P_1},\ldots,\frac{p_{(KN_K)}}{1-P_K}\right).
\end{equation}

%
\par \medskip
Let us end Example 4 by considering two special cases of the proportion-based reward function \re{Rgkab}. The first one occurs when $a=b=c$, and it corresponds to a reward function 
\beq
R(\cI,i) = 1(i\in\cI) - c\left[|\cI|-1(|\cI_{k(i)}|>0)\right]
\lb{RcMany}
\eeq
that differs slightly from \re{Rc} in that all categories in the classified set $\cI$ are penalized by $c$ when none of them belong to the correct block $k(i)$, that is, when $\cI_{k(i)}=\emptyset$.
\par\medskip
The second special case of \re{Rgkab} occurs when it is known that observation $z$ belongs to block $k$, i.e.\ $P_k(z)=1$. The classifier $\hI_{\sbm(z)}$ in \re{Ibnz}, with $\bm(z)=(m_1(z),\ldots,m_K(z))$ as in \re{nkzab}, then simplifies to 
$$
\begin{array}{rcl}
m_k(z) &=& \max\left\{1,\max\{2\le m_k \le N_k; \, p_{(k,N_k+1-m_k)} \ge a\}\right\},\\
m_l(z) &=& 0, \,\, l\ne k.
\end{array}
$$
This corresponds to an $m$-MP classifier \re{hInz}, with $m=m(z)$ as in \re{nzLin} and \re{nzLin2}, when classification is restricted to categories within class $k$, and a penalty $c=a$ is incurred per extra classified category. 
\hfill\slut
\par\medskip
{\bf Example 5} (Indifference zones.) Assume that we want to know which of $N$ normally distributed populations with unit variances and expected values $\theta_1,\ldots,\theta_{N}$ has the largest expected value $\theta_{(N)}$. Let 
$$
Z=(Z_{ij}; \, 1\le i \le N, 1\le j \le n_i)
$$
be a sample of independent random variables, with $Z_{ij}\sim N(\theta_i,1)$. Letting $\phi$ be the density function of a standard normal distribution, this gives a likelihood 
$$
f(z;\theta) = \prod_{i=1}^N \prod_{j=1}^{n_i} \phi(z_{ij}-\theta_i)
$$
with a parameter vector $\theta=(\theta_1,\ldots,\theta_{N})$ that is assumed to have a prior density $P(\theta)$. Divide the parameter space $\Theta = \R^{N}$ into a disjoint union 
\beq
\Theta = \Theta_1 \cup \ldots \cup \Theta_{N+1}
\lb{Theta}
\eeq
of $N+1$ regions, where
\beq
\begin{array}{rcl}
\Theta_i &=& \{\theta; \, \theta_i = \theta_{(N)} \ge \theta_{(N-1)} + \eps\}, \quad i=1,\ldots,N,\\
\Theta_{N+1} &=& \Theta \setminus (\Theta_1\cup \ldots \cup \Theta_{N})
\end{array}
\lb{Thetai}
\eeq
and $\eps>0$ is a small number. Whereas $\Theta_1,\ldots,\Theta_{N}$ correspond to all hypo-theses where some parameter $\theta_i$ is largest by a margin of at least $\eps$, $\Theta_{N+1}$ is the {\it indifference zone}, where none of the populations has an expected value that is at least $\eps$ units larger than all the others (Bechhofer, 1954, Goldsman, 1986).  
\par\medskip
It is possible to put this model into the framework of Section \ref{Sec:ClassSetRew}, with 
\beq
\begin{array}{rcl}
\pi_i &=& \int_{\Theta_i} P(\theta)d\theta,\\
f_i(z) &=& \int_{\Theta_i} f(z;\theta)P(\theta)d\theta / \int_{\Theta_i} P(\theta)d\theta,
\end{array}
\lb{piifi}
\eeq
for $i=1,\ldots,N+1$, and an extra category $N+1$ is added that represents the indifference zone. We will consider reward functions 
\beq
R(\cI,i) = \left\{\begin{array}{ll}
1(\cI=\{i\}), & i=1,\ldots,N,\\
r 1(\cI=\emptyset), & i=N+1,
\end{array}\right.
\lb{RIndZone}
\eeq
where $r>0$ is the reward of not selecting any population as having the largest mean, when the parameter vector belongs to the indifference zone. Formally, this corresponds to a block invariant reward function with the two blocks
\beq
\begin{array}{rcl}
\cN_1 &=& \{1,\ldots,N\},\\
\cN_2 &=& \{N+1\}
\end{array}
\lb{cN1cN2}
\eeq 
of categories, although misclassification in this example is more serious {\it within} than between blocks. However, the reward function \re{RIndZone}, with blocks as in \re{cN1cN2}, does not belong to the class of reward functions in \re{Rgk}. It is therefore not possible to make use of Propositon \ref{Prop:Several}, but it can still be seen that the optimal classifier is 
$$
\hI = \left\{\begin{array}{ll}
(1N); & p_{(1N)}\ge r p_{N+1},\\
\emptyset; & p_{(1N)} < rp_{N+1}.
\end{array}\right.
$$
In Section \ref{Sec:Disc} we briefly discuss an extended class of reward functions that includes \re{RIndZone}. 
\hfill\slut  

\section{Illustration for an ornithological data set}\lb{Sec:Data}
In this section taxon identification will be used as a case study. In particular, we will look at four bird species that are morphologically similar, but share three measurable traits. In \citet{karlsson2021identification}, we treat the under\-lying fitting problem in detail, with covariates, heteroscedasticity, missing values and imperfect observa\-tions, examplified on data on the same four bird species. Since this paper has a stronger emphasis on developing a theory of classification, we use a subset of data from \citet{karlsson2021identification} for the purpose of illustration. This data set includes complete observations only from a certain stratum of the population, elimina\-ting the need for covariates. To a large extent, our data is the same as in \citet{Malmhagen2013}, where the same classification problem is treated with mostly descriptive statistics.

\subsection{The data set}
\begin{table}[ht]
\centering
\begin{tabular}{rr}
  \hline
 species & $n_i$ \\ 
  \hline
Reed warbler & 409 \\ 
  Blyth's reed warbler &  41 \\ 
  Paddyfield warbler &  18 \\ 
  Marsh warbler & 414 \\ 
   \hline
\end{tabular}
\caption{The number of observations of each species.}
\label{datatab}
\end{table}

The four species considered are \emph{Reed warbler}, \emph{Marsh warbler}, \emph{Blyth's reed warbler} and \emph{Paddyfield warbler}, and the three shared traits are \emph{wing length}, \emph{notch length} and \emph{notch position}. For details on the measurements of the traits, see e.g. \citet{svensson1992identification, Malmhagen2013} and \citet{karlsson2021identification}. In total we have 882 complete observations of juvenile birds, with the number $n_i$ of birds of each species given in Table \ref{datatab}. This gives rise to a training data set 
$$
\cD_i = \{z_{ij}; \, j=1,\ldots,n_i\}
$$
for each species $i = 1,\ldots,4$.
\par\medskip
The species were partitioned as follows: \emph{Reed warbler} and \emph{Marsh Warbler} constitute the block ``common breeders'', \emph{Blyth's reed warbler} constitutes the block ``rare breeder'' and \emph{Paddyfield warbler} constitutes the block ``rare vag-rant''. We acknowledge that the grouping is quite arbitrary and that the block names are inaccurate in most places of the world, but here it will mainly illustrate classification with a partitioned label space. 

\subsection{Model}\lb{Sec:Model}
We assume that the parameters associated with each category are indepen\-dent. If $\theta_i\in\Theta_i$ is the parameter vector associated with category $i$, with a prior distribution $P(\theta_i)$, and if $f(\cD_i;\theta_i)$ refers to the likelihood of training data for taxon $i$, the posterior distribution of $\theta_i$ is  
\beq
P(\theta_i|\cD_i) = \frac{P(\theta_i)f(\cD_i|\theta_i)}{P(\cD_i)}.
\lb{PostThetai}
\eeq
Let $z$ be an observed new data point that we wish to classify, based on training data. If $f(z;\theta_i)$ is the likelihood of the test data point for taxon $i$, we integrate over the posterior distribution \re{PostThetai} in order to obtain the corresponding likelihood  
\beq
f_i(z) = \int_{\Theta_i} f(z;\theta_i)P(\theta_i|\cD_i)d\theta_i
\lb{fiz}
\eeq
of $z$ for the hypothesis $H_i$ that corresponds to this taxon. Then we insert \re{fiz} into \re{pi} in order to obtain the posterior probabilities $p_i(z)$ of all categories. In our example we assume that $z_{ij} \sim \text{MVN}\left(\mu_i, \Sigma_i\right)$ has a multivariate normal distribution $f(\cdot;\theta_i)$, with $\theta_i = (\mu_i,\Sigma_i)$. We estimate \eqref{fiz} using Monte Carlo simulation, i.e.\ we simulate $L$ realizations of $\theta_{il} = \left(\mu_{il}, \Sigma_{il}\right)$ from $P(\theta_i \mid \cD_i)$ for each $i$, compute
\beq
\hat{f}_i(z) = \frac{1}{L} \sum_{l=1}^L f(z;\theta_{il})
\lb{fizq}
\eeq
and then plugg \eqref{fizq} into \eqref{pi}. For a more detailed model setup, see Section 3.1 and Appendix A of \citet{karlsson2021identification}. All implementation was done in \textsf{R} \citep{r2021}, using the package \texttt{mvtnorm} \citep{mvtnorm}.

\subsection{Classifiers based on composite proportion-based re\-ward functions}\lb{Sec:CVProc}

We will derive composite classifiers from the composite proportion-based reward function \re{Rgkab}. This reward function involves the two constants $a\ge 0$ and $b>0$, and the corresponding classifier of $z$ is denoted $\hI_{(a,b)}(z)$. We will regard $0 \le \va = a/b$ as a fixed parameter that quantifies how much more severe it is to misclassify a category outside a block than inside it, with severity inversly proportional to $\varepsilon$. If $0\le \va < 1$, it is more severe to misclassify between blocks than within, whereas the opposite is true when $\va>1$. The parameter $b$ will be chosen through leave-one-out cross validation. To this end, let $\tilde{R}(\cI,i)$ be a binary-valued reward function (to be chosen below) without any penalty term, and let 
\begin{equation}
R_{ab}^{\scr{cv}}(\tilde{R}) = \sum_{i=1}^N \frac{w_i}{n_i} \sum_{j=1}^{n_i} \tR(\hI_{(a,b)}(z_{ij}),i)
\lb{Rbcv}
\end{equation}
refer to the fraction of observations $z_{ij}$ in the training data set $\cD = \cD_1 \cup \ldots \cup \cD_N$ that return a reward in the cross-validation procedure. That is, $R_{ab}^{\scr{cv}}(\tilde{R})$ equals the fraction of observations $z_{ij}$ with $\tR(\hI_{(a,b)}(z_{ij}),i)=1$, where $\hI_{(a,b)}(z_{ij})$ is the classifier of $z_{ij}$ based on the rest of the data.      
It is further assumed that $w_i$ are non-negative weights that sum to $1$, such as $w_i=1/N$ or $w_i = n_i/\sum_{j=1}^N n_j$. 
\par\medskip
The choice of $b$ will depend on which reward function $\tR$ that is used in \re{Rbcv}. 
Some possible choices of the binary-valued reward function are given in Table \ref{tab:reward}.
For two of them ($\tR_3$ and $\tR_4$) the reward $\tR(\cI,i)$ is a non-decreasing function of $\cI$, and therefore the non-reward rate $1-R_{\va b,b}^{\scr{cv}}$, obtained from the cross validation procedure \re{Rbcv}, is a non-decreasing function of $b$. We will therefore choose $b$ as the largest cost parameter for which the non-reward rate is at most $\de>0$, i.e.
\begin{equation}
b_{\va\de} = \max\{b\ge 0; \, 1-R_{\va b,b}^{\scr{cv}} \le \de\}.
\lb{b}
\end{equation}
In particular, it can be shown that when the reward function $\tR_3$ is used in \eqref{b}, this generalizes the conformal algorithm \citep[Section 4.3]{shafer2008tutorial} from the case of one block ($K=1$) to several blocks ($K>1$), although we use cross-validation from training data rather than prediction of a new observations, as in \citet{shafer2008tutorial}.
\par\medskip
For the other two reward functions $\tR_1$ or $\tR_2$ of Table \ref{tab:reward}, it is no longer the case that estimated non-reward rate $1 - R_{\va b,b}^{\scr{cv}}$ in \re{Rbcv} is monotonic in $b$. For these two choices of $\tR$ we rather choose the cost parameter
\beq
b_\va = \arg\min_{b\ge 0} (1 - R_{\va b,b}^{\scr{cv}}) 
\lb{b2}
\eeq
in order to minimize the estimated non-reward rate. 

\begin{table}
 \centering
 \def\arraystretch{1.5}
 \begin{tabular}{|l|l|}
 \hline
 \textit{Binary reward function} & \textit{Reward critera}\\
 \hline
  $\tR_1(\cI,i) = 1(\cI=\{i\})$ & Correct (point) classification. \\ \hline
  $\tR_2(\cI,i) = 1(i\in\cI \land \cI \subseteq \cN_{k(i)})$ & \makecell[l]{Correct category is in the classifier, \\and no category from an incorrect \\ block is in the classifier. \\ The analogy of $\tR_1$ for block prediction.} \\ \hline
  $\tR_3(\cI,i) = 1(i\in\cI)$ & Correct category in classifier. \\ \hline
  $\tR_4(\cI,i) = 1(\cI\cap\cN_{k(i)}\ne\emptyset)$ & \makecell[l]{Some category from the correct block \\ in the classifier. \\ The analogy of $\tR_3$ for block prediction.} \\ \hline
 \end{tabular}
\caption{The four binary reward functions used in the case study. Since $\tR_1(\cI,i) \le \tR_2(\cI,i) \le \tR_3(\cI,i) \le \tR_4(\cI,i)$, it follows that $\tR_1$ is the least generous reward function and $\tR_4$ the most generous one. Note that $R_{\va b,b}^{\scr{cv}}(\tilde{R}_3)$ and $R_{\va b,b}^{\scr{cv}}(\tilde{R}_4)$ are both decreasing functions of $b$, so that \re{b} makes sense for choosing $b$ for any of these two choices of $\tR$, whereas \re{b2} is more approriate for choosing $b$ for the other two reward functions.}
\label{tab:reward}
\end{table}

\par\medskip
We will look at three different prior distributions, namely a uniform prior ($\piflat$), a prior proportional to the number of observations $n_i$ from each category in training data ($\piprop$), and a prior proportional to the number of registered birds of each species at the Falsterbo Bird observatory throughout its operational history ($\pireal$). The purpose of $\piflat$ is to represent a situa\-tion of no prior knowledge on how likely any of the categories is to occur. The prior $\pireal$ is supposed to examplify a real world situation of having some commonly occuring species and some very rare ones, whereas $\piprop$ is a middle ground between these two extremes that fits the data set $\cD$ very well. As weights we choose 
\begin{align}
  \wbird_i &= n_i/\sum_{j=1}^4 n_j, \nonumber\\
  \wspec_i &= 1/4, \nonumber\\
  \wrare_i &= 1 / \sum_{j=1}^4 \pi_i^{(\text{real})}/{\pi_j^{(\text{real})}},
\label{weights2}
\end{align}
for $i=1,2,3,4$. We weight each observation equally with $\wbird$, each species equally with $\wspec$, whereas the species are weighted higher the less expected they are with $\wrare$. Since the number of observations are not balanced across species, $\wbird$ will weight species higher the more common they are, i.e.\ the more observations we have of them. Using $\wspec$, birds will be weighted unequally due to the same imbalance (the less birds of a given species there are, the higher weights are assigned to these birds). Finally, as mentioned above, $\wrare$ weights less frequently observed species more heavyly; the rationale being that we value observations of rarely occuring species, as data on these are scarce.

\subsection{Results}

We will analyze the estimated reward rate \eqref{Rbcv} for the four choices of $\tilde{R}$ that are listed in Table 2. In Tables \ref{Tab:bdelta} and \ref{Tab:bopt} we present the automatic choice of the cost parameter $b$ (cf.\ \re{b} and \re{b2}) for two ratios $\va = 1/2$ and $\va = 2$ of $a$ and $b$, and for all nine combinations priors and weights. As can be seen from Table \ref{Tab:bopt}, the same optimal $b$-values are found for $\tR_1$ and $\tR_2$, with the same non-reward rates. This is mostly due to the small block sizes, meaning that it sometimes is equivalent to pick the correct species, as to pick the correct block.
\par\medskip
In Figures \ref{bbig} and \ref{abig} we plot the value of estimated non-reward rate $1-R_{\va b,b}^{\scr{cv}} $ for a grid of $b$-values, for $\va=1/2$ and $\va=2$ respectively. Note the monotone decrease of $\tR_3$ and $\tR_4$ as $b$ decreases. Also notice the minimums of $\tR_1$ and $\tR_2$ in the graphs.
\par\medskip
Finally we refer to Appendix \ref{app:2d} for further visualisations of $\tR_1$ and $\tR_2$, evaluated over a lattice of $a$ and $b$-values. It can seen that for $\tR_2$ the optimal non-reward rate is achived by choosing $a=0$. This is straightforward to explain, as $\tR_2$ does not punish the inclusion of several categories from the correct block. For this reason the classifier $\hI_{(a,b)}$ with minimal non-reward rate includes as many categories as possible from each block with at least one classified member, corresponding to $a=0$. Notice also that there are large regions of values of $a$ and $b$ that attain the minimum non-reward rate.

\begin{table}
\centering
\footnotesize
\def\arraystretch{1.3}
 \begin{tabular}{|c|c|c|c|c|c|}
 \hline
  \multirow{2}{*}{$\va$} & \multirow{2}{*}{Prior}& \multirow{2}{*}{$\tilde{R}$} & \multicolumn{3}{c}{$b_{\va,0.05}$} \vline \\ \cline{4-6}
  & & & $\wbird$ & $\wspec$ & $\wrare$ \\
  \hline
\multirow{6}{*}{$1/2$} & \multirow{2}{*}{$\piflat$} & $\tR_3$ & $\ge 20.00$ & 2.29 & 2.11 \\ 
   &  & $\tR_4$ & $\ge 20.00$ & 3.05 & 2.11 \\ \cline{2-6}
  & \multirow{2}{*}{$\piprop$} & $\tR_3$ & $\ge 20.00$ & 5.23 & 4.81 \\ 
  &  & $\tR_4$ & $\ge 20.00$ & 6.96 & 4.81 \\ \cline{2-6}
  & \multirow{2}{*}{$\pireal$} & $\tR_3$ & 1.07 & 0.43 & 0.18 \\ 
  &  & $\tR_4$ & $\ge 20.00$ & 0.52 & 0.18 \\ \cline{1-6}
  \multirow{6}{*}{$2$} & \multirow{2}{*}{$\piflat$} & $\tR_3$ & $\ge 20.00$ & 2.29 & 2.11 \\ 
   &  & $\tR_4$ & $\ge 20.00$ & 3.05 & 2.11 \\ \cline{2-6}
   & \multirow{2}{*}{$\piprop$} & $\tR_3$ & $\ge 20.00$ & 5.23 & 4.81 \\ 
   &  & $\tR_4$ & $\ge 20.00$ & 6.96 & 4.81 \\ \cline{2-6}
   & \multirow{2}{*}{$\pireal$} & $\tR_3$ & 1.07 & 0.21 & 0.18 \\ 
   &  & $\tR_4$ & $\ge 20.00$ & 0.52 & 0.18 \\ 
  \hline
 \end{tabular}
 \caption{The table specifies the estimated values of the cost parameter $b$, using \eqref{b} with $\delta=0.05$. These estimates of $b$ are computed for each combination of $\va$, prior $\pi_i$, and weights $w_i$. We evaluated \eqref{Rbcv} for $0.01 \le b \le 20$ with a resolution of 0.01.}
 \label{Tab:bdelta}
\end{table}

\begin{table}
\centering
\footnotesize
\def\arraystretch{1.5}
 \begin{tabular}{|c|c|c|c|c|c|c|c|c|}
 \hline
  \multirow{2}{*}{$\va$} & \multirow{2}{*}{Prior}& \multirow{2}{*}{$\tilde{R}$} & 
  \multicolumn{2}{c}{$\wbird$} \vline & \multicolumn{2}{c}{$\wspec$} \vline & \multicolumn{2}{c}{$\wrare$} \vline \\ \cline{4-9}
  & & & $b_\va$ & \makecell[c]{non-reward \\ rate} & $b_\va$ & \makecell[c]{non-reward \\ rate} & $b_\va$ & \makecell[c]{non-reward \\ rate} \\
  \hline
\multirow{6}{*}{$1/2$} & \multirow{2}{*}{$\piflat$} & $\tR_1$ & 1.24 & 1.59\% & 17.16 & 9.12\% & 1.16 & 4.06 \% \\ 
& & $\tR_2$ & 1.24 & 1.59\% & 17.16 & 9.12\% & 1.16 & 4.06\% \\ \cline{2-9}
& \multirow{2}{*}{$\piprop$} & $\tR_1$ & 2.63 & 1.59\% & 2.24 & 4.06\% & 2.63 & 4.06\% \\
&  & $\tR_2$ & 2.63 & 1.59\% & 2.24 & 4.06\% & 2.63 & 4.06\% \\ \cline{2-9}
& \multirow{2}{*}{$\pireal$} & $\tR_1$ & 10.28 & 6.69\% & 10.28 & 19.41\% & 10.28 & 49.62\% \\
&  & $\tR_2$ & 10.28 & 6.69\% & 10.28 & 19.41\% & 10.28 & 49.62\% \\
  \hline
\multirow{6}{*}{$2$} & \multirow{2}{*}{$\piflat$} & $\tR_1$ & 1.24 & 1.59\% & 17.16 & 9.12\% & 1.16 & 4.06\% \\ 
& & $\tR_2$ & 1.24 & 1.59\% & 17.16 & 9.12\% & 1.16 & 4.06\% \\ \cline{2-9}
& \multirow{2}{*}{$\piprop$} & $\tR_1$ & 2.63 & 1.59\% & 2.24 & 4.06\% & 2.63 & 4.06\% \\
&  & $\tR_2$ & 2.63 & 1.59\% & 2.24 & 4.06\% & 2.63 & 4.06\% \\ \cline{2-9}
& \multirow{2}{*}{$\pireal$} & $\tR_1$ & 10.28 & 6.69\% & 10.28 & 19.41\% & 10.28 & 49.62\% \\
&  & $\tR_2$ & 10.28 & 6.69\% & 10.28 & 19.41\% & 10.28 & 49.62\% \\ \hline
 \end{tabular}
 \caption{The table specifies the estimated values of the cost parameter $b$, using \re{b2}. These estimates of $b$ are computed for each combination of $\va$, prior $\pi_i$, and weights $w_i$. They were found using the \texttt{optimise}-function in \textsf{R}.}
 \label{Tab:bopt}
\end{table}

\begin{figure}
 \centering
 \caption{\footnotesize This figure represents the case $\va = 1/2$. The prior is specified in the title of each graph, whereas the weights are explained in the subcaptions. Each color of the functions in the graphs correspond to one of the four reward functions $\tR_1,\tR_2,\tR_3,\tR_4$, given by the legends above each subfigure.
 For all priors and weights we observe that $\tR_3$ and $\tR_4$ decrease monotonically as $b$ decreases, whereas $\tR_1$ and $\tR_2$ have a global minimum.}
     \begin{subfigure}[b]{\textwidth}
         \centering
         \caption{\footnotesize Using $\wbird$ we observe similar curves for $\piflat$ and $\piprop$, whereas $\pireal$ gives overall higher non-reward rates.}
         \includegraphics[width=\textwidth]{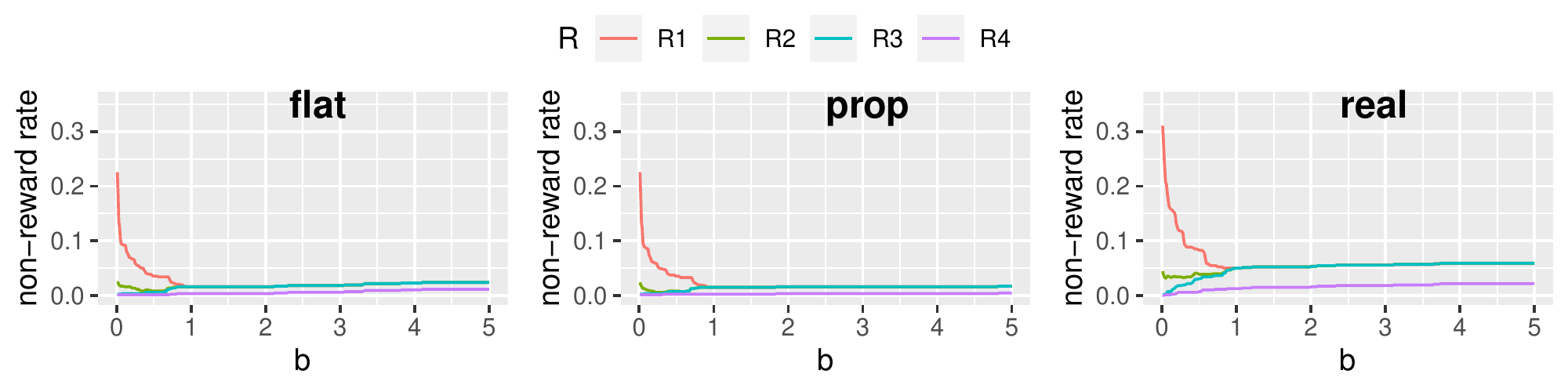}
     \end{subfigure}
     \vfill
     \begin{subfigure}[b]{\textwidth}
         \centering
         \caption{\footnotesize Using $\wspec$, the curves for $\tR_1$ and $\tR_2$ look similar for $\piflat$ and $\piprop$, wheras the large drop in $\tR_3$ and $\tR_4$ occurs either for large values of $b$ ($\piflat$) or for small values of $b$ ($\piprop$). Again, the non-reward rates are overall higher using $\pireal$.}
         \includegraphics[width=\textwidth]{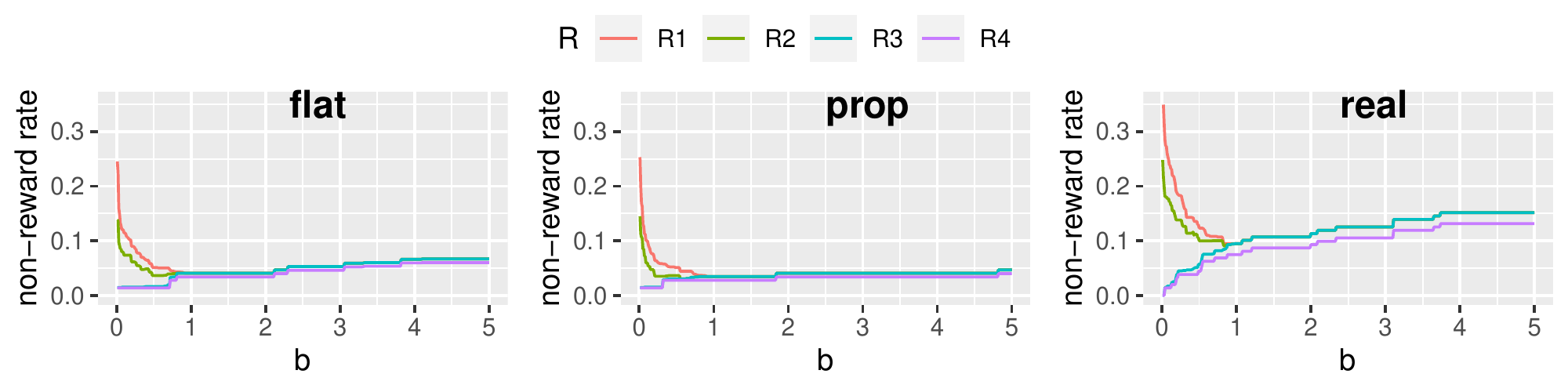}
     \end{subfigure}
     \vfill
     \begin{subfigure}[b]{\textwidth}
         \centering
         \caption{\footnotesize The scale along the vertical axis (for the non-reward rate) is different in these subplots compared to (a) and (b). Using $\wrare$, with a very uneven weighting of species, the curves for $\tR_1$ and $\tR_2$, as well as for $\tR_3$ and $\tR_4$, take values very close to each other. Thus it seems like these curves overlap, when in reality they do not. This is just a consequence of the extreme amount of up-weighting of rare species, which are both in singelton blocks.}
         \includegraphics[width=\textwidth]{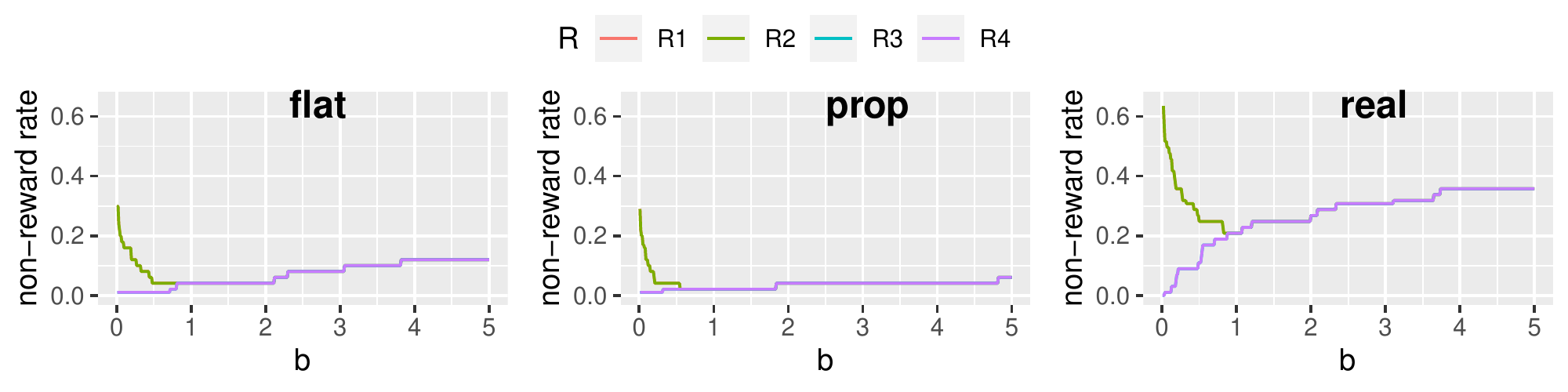}
     \end{subfigure}
        \label{bbig}
\end{figure}

\begin{figure}
 \centering
 \caption{\footnotesize This figure represents the case $\va = 2$. The prior is specified in the title of each graph, whereas the weights are explained in the subcaptions. Each color of the functions in the graphs correspond to one of the four reward functions $\tR_1,\tR_2,\tR_3,\tR_4$, given by the legends above each subfigure.
 For all priors and weights we observe that $\tR_3$ and $\tR_4$ decrease monotonically as $b$ decreases, whereas $\tR_1$ and $\tR_2$ have a global minimum.}
     \begin{subfigure}[b]{\textwidth}
         \centering
         \caption{\footnotesize Using $\wbird$ we observe similar curves for $\piflat$ and $\piprop$, whereas $\pireal$ gives overall higher non-reward rates.}
         \includegraphics[width=\textwidth]{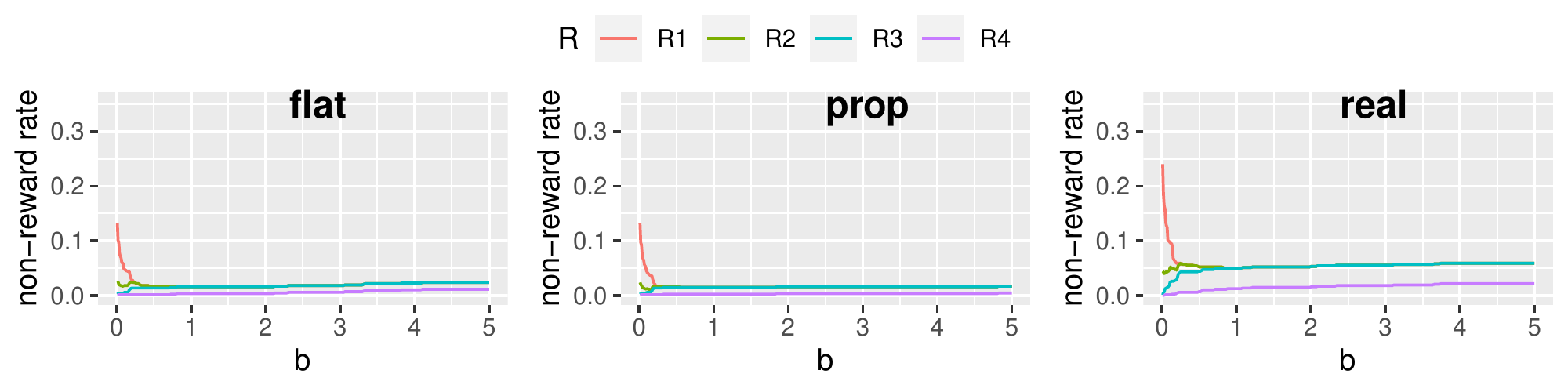}
     \end{subfigure}
     \vfill
     \begin{subfigure}[b]{\textwidth}
         \centering
         \caption{\footnotesize Using $\wspec$, the curves for $\tR_1$ and $\tR_2$ look similar for $\piflat$ and $\piprop$, wheras the large drop in $\tR_3$ and $\tR_4$ occurs either for large values of $b$ ($\piflat$) or for small values of $b$ ($\piprop$). Again, the non-reward rates are overall higher using $\pireal$.}
         \includegraphics[width=\textwidth]{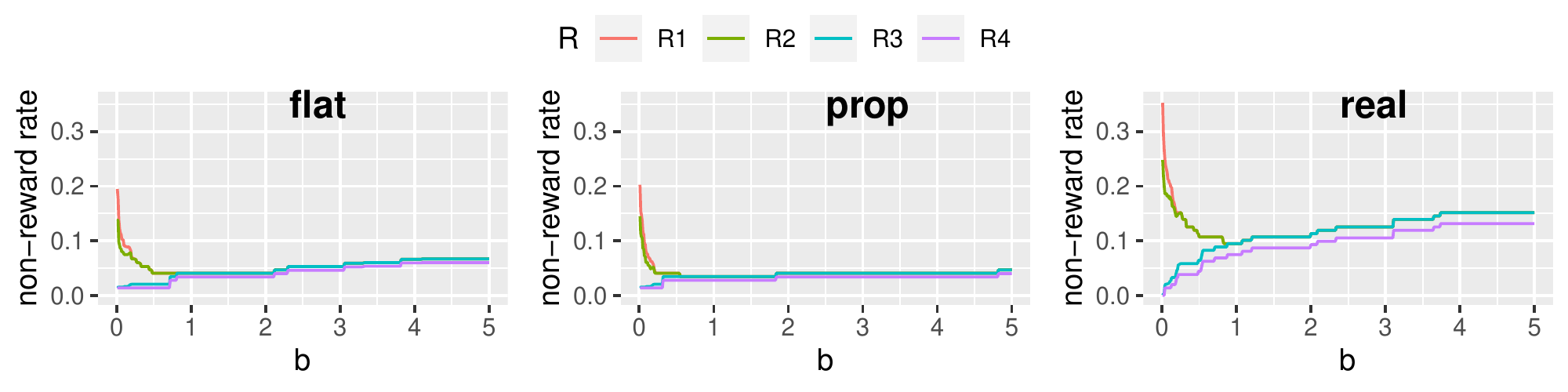}
     \end{subfigure}
     \vfill
     \begin{subfigure}[b]{\textwidth}
         \centering
         \caption{\footnotesize The scale along the vertical axis (for the non-reward rate) is different in these subplots compared to (a) and (b). Using $\wrare$, with a very uneven weighting of species, the curves for $\tR_1$ and $\tR_2$, as well as for $\tR_3$ and $\tR_4$, take values very close to each other. Thus it seems like these curves overlap, when in reality they do not. This is just a consequence of the extreme amount of up-weighting of rare species, which are both in singelton blocks.}
         \includegraphics[width=\textwidth]{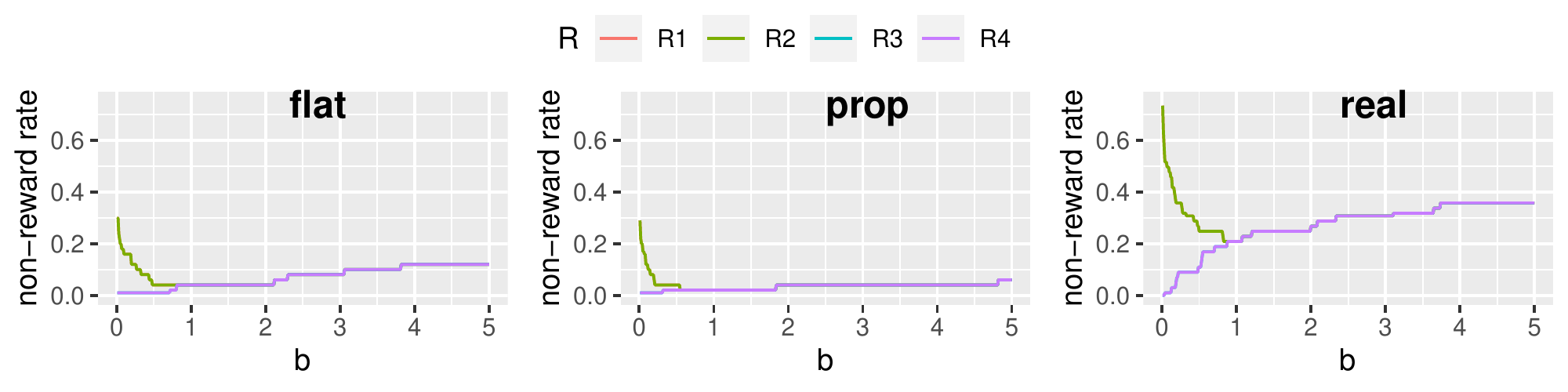}
     \end{subfigure}
        \label{abig}
\end{figure}

\newpage
\section{Discussion}\lb{Sec:Disc}

In this article we introduce a general framework of set-valued classification of data that originates from one of a finite number of possible hypotheses. We introduce reward functions with a set-valued input argu\-ment, and derive the optimal (Bayes) classifier by maximizing the expected reward. Explicit formulas for the Bayes classifier are derived for a large class of reward func\-tions. This includes scenarios where hypotheses either consitute one homo\-geneous block, or consist of several blocks of hypotheses, such that ambiguity within blocks is (typically) less serious than ambiguity between them. Our procedure is illustrated with an ornithological data set, where taxa (hypo\-theses) are divided into blocks. 
\par\medskip
As mentioned in \citet{ripley2007pattern}, a possible reason for including reject options is to obtain classifiers that are more reliable but also less expensive to use. In our case study of Section \ref{Sec:Data}, for instance, a possible option is to consult an expert who would be able to identify the bird species morphologically, without using the measured traits. Although expertise does not come cheap, this could still be an alternative when the expected cost of algorithmic classification exceeds the cost of consulting an expert. The latter cost might be independent or a function of the number of hypotheses she needs to consider. In the former case, the reject option of \citet{ripley2007pattern} would suffice, and in the latter case a partial rejection, as proposed in this paper, could be beneficial.
\par\medskip
A number of generalizations of our work is possible. Firstly, suppose the new observation $z$ that we want to classify, by means of the optimal classifier $\hI=\hI(z)$ in \re{Vz}, involves a covariate vector $x$ and a reponse variable $Y$. Following \citet{karlsson2021identification}, the most straightforward approach is to include covariates into the observation $z=(x,Y)$ that is to be classified. The covariate information will then be included in the category distributions $f_i$ and posterior probabilities $p_i$ of all categories $i=1,\ldots,N$, as well as in the resulting $m$-MP and composi\-te  classifiers. However, if we also want the ambiguity of the classifier to depend on covariate information, it is possible to let the cost parameters of the reward function $R$ depend on $x$ as well. For instance, the conformal prediction algorithm of Example 3 involves choosing the cost $c(\de) = F^{-1}(\delta)=Q(\de)$ of including more lables in the classifier, as a quantile $Q(\delta)$ of the distribution $F$ of posterior probabilities (cf.\ \eqref{quant}). In this context it is possible to use quantile regres\-sion \citep{koenker2005quantile,bottai2010} and choose the cost parameter $c(\de;x)=Q(\delta | x) = g^{-1}[x^\top \beta(\delta)]$ as a conditional quantile function. This is a regression model where the parameter vector $\beta(\delta)$ is a function of the quantile $\delta$, whereas $g$ is a link function. This approach might be particularly helpful for models with heteroscedasticity.
\par\medskip
Secondly, it is possible to consider reward functions with other penalty terms. Instead of penalizing the number of classified categories within the correct and wrong blocks respectively, as in \re{Rgk}, one penalizes the number of {\sl wrongly classified categories} within the correct and wrong blocks. This corresponds to a reward function  
\beq
R(\cI,i) = 1(i\in \cI) - g_{k(i)}(|\cI_{k(i),(-i)}|,|\cI|- |\cI_{k(i)}|,|\cI|),
\lb{RSeveral}
\eeq
where $\cI_{k(i),(-i)} = \cI_{k(i)}\setminus \{i\}$ is the number of wrong categories of $\cI$ that belong to block $k(i)$ when category $i$ is true. For instance, the reward function \re{RIndZone} for indifference zones (Example 5) can be formulated in accordance with \re{RSeveral}. In order to see this let $m_1$ and $m_2$ refer to the number of wrongly classified categories of blocks 1 and 2, whereas $m$ is the size of the classified set (hence $m-(m_1+m_2)$ equals 1 or 0 depending on whether $i\in\cI$ or not). Then \re{RIndZone} corresponds to having penalty terms 
$$
\begin{array}{rcl}
g_1(m_1,m_2,m) &=& 1(m=m_1+m_2+1\mbox{ and }(m_1,m_2)\ne (0,0)),\\
g_2(0,m_1,m) &=& 1(m=m_1+1) -r1(m=m_1\mbox{ and }m_2=0),
\end{array}
$$
when the true category belongs to block 1 and 2 respectively.  
\par\medskip
Thirdly, suppose an observation $z$ belongs to {\it several} categories $\cJ\subset\cN$, such as when $\cJ$ represents the properties associated with $z$. It is natural in this context to have a reward $R(\cI,\cJ)$ for a classified set $\cI\subset\cN$ of categories. The MAP classifier \re{RMAP}, for instance, corresponds a reward function $R(\cI,\cJ)=1(\cI=\cJ)$. If we allow for some misclassified categories and all categories belong to one homogeneous block, a natural extension of \re{Rg} is a function $R(\cI,\cJ)=1(\cJ\subset\cI) - g(|\cI|)$ where the first term gives a unit reward if all true categories are included in the classifier, whereas the second term $g(|\cI|)$ penalizes the size $|\cI|$ of the classified set.

\section*{Acknowledgements}
The authors would like to thank Vilhelm Niklasson at Stockholm University for suggesting the topic of conformal prediction, and the Falsterbo Bird Observatory for providing the data set.

%

\appendix

\section{Optimising $\tR_1$ and $\tR_2$} \label{app:2d}

In Figure \ref{fig:2d} the non-reward rate $1-R_{ab}^{\scr{cv}}(\tilde{R}) $ (cf.\ \eqref{Rbcv}) is plotted as a function of the two cost parameters $a$ and $b$ of the classifier $\hI_{(a,b)}$ for the two reward functions $\tR_1$ and $\tR_2$ of Table \ref{tab:reward}. The objective function is not smooth and thus it can be hard to optimize. However, we obtained good results with Nelder-Mead optimization \citep{nelder1965simplex}, as implemented in the \texttt{optim}-function in \textsf{R}, with a starting value of $(a,b)$ that corresponds to a small non-reward rate. 
\begin{figure}
\centering
\caption{The figure contains filled contour plots of the estimated non-reward rate $1-R_{ab}^{\scr{cv}}(\tilde{R}) $ (cf.\ \eqref{Rbcv}), as a function of the two cost parameters $a$ and $b$ of the classifier $\hI_{(a,b)}$. These estimated non-reward rates make use of weights $w_i^{(\scr{bird})}$ (cf.\ \re{weights2}), the reward function $\tR_1$ (top row) and $\tR_2$ (bottom row). The columns, from left to right, correspond to the priors $\piflat$ , $\piprop$ and $\pireal$ (cf. Section \ref{Sec:CVProc}). The levels of the contour plots are crudely drawn, but it can still seen that $\tR_1$ attains a low value for a large set of $(a,b)$,  whereas $\tR_2$ attains its lowest values over a small region where $a$ is close to 0.}
 \includegraphics[width=\textwidth]{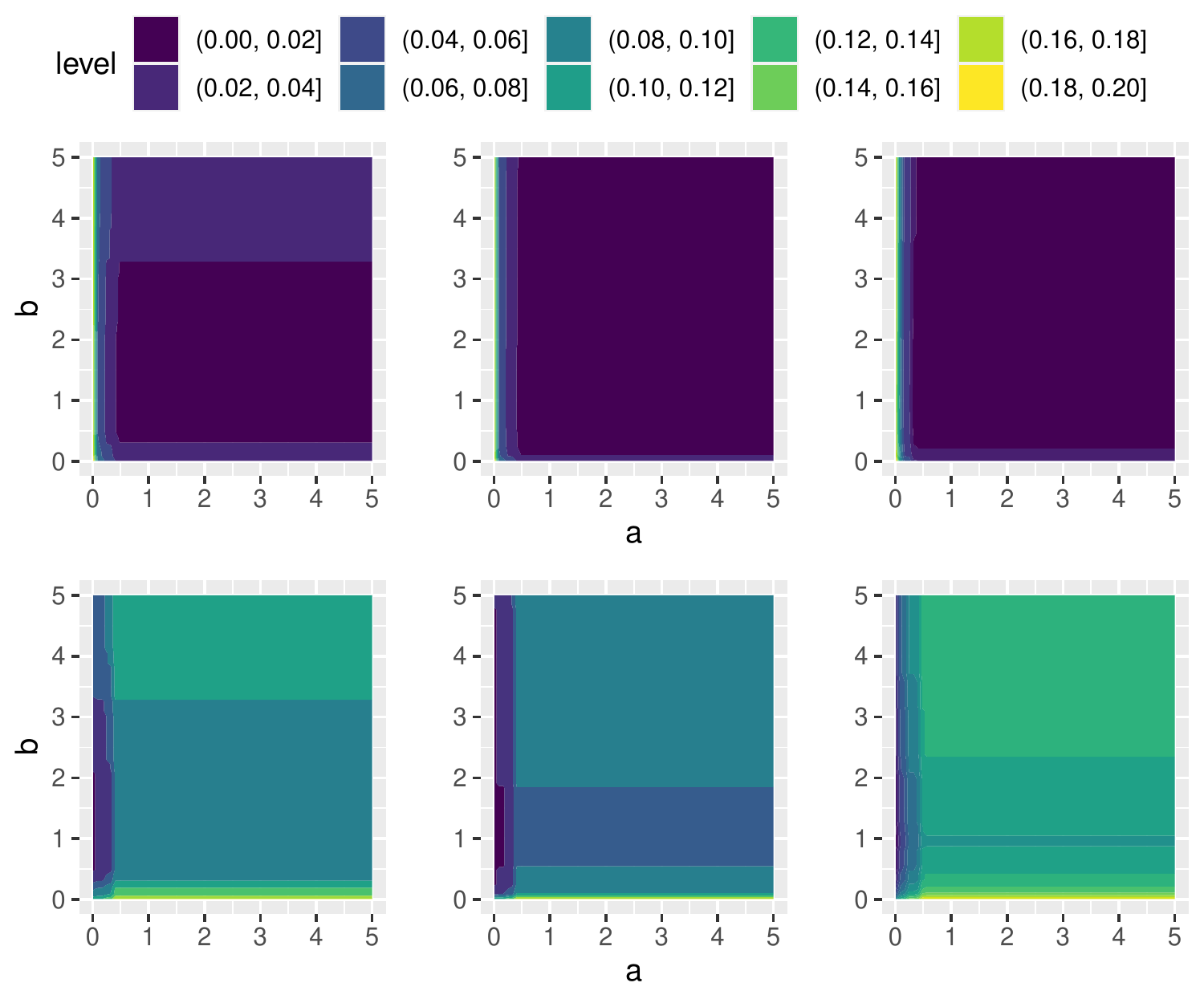}
 \label{fig:2d}
\end{figure}

\newpage
\bibliography{refs}

\end{document}